\documentclass[]{theclass} 
\usepackage{graphicx, xfrac, lineno, float, subcaption, tasks, comment, xcolor, booktabs, multirow}
\usepackage[T1]{fontenc}
\usepackage[normalem]{ulem}
\usepackage{datetime}
\usepackage{colortbl}
\usepackage{enumitem}
\usepackage{datetime}

\settasks{
	counter-format=(tsk[r]),
	label-width=4ex
}

\begin{document}
\begin{frontmatter}

\titledata{A tris of perfect matchings in\\ bridgeless claw-free cubic graphs}{}  

\authordata{Jean Paul Zerafa}
{Department of Mathematics\\ L-Universit\`{a} ta' Malta\\ Msida, MSD 2080, Malta}
{jean-paul.zerafa@um.edu.mt}{}


\begin{abstract}
A proof of the cycle double cover conjecture was recently announced, yielding an $8$-cycle double cover for every bridgeless graph. The stronger $5$-cycle double cover conjecture, which is still open, is equivalent to the statement that the edge set of every bridgeless claw-free cubic graph can be covered by at most four perfect matchings. Perfect matchings in bridgeless cubic graphs have been studied extensively, with two of the main conjectures in this area being the Berge--Fulkerson and the Fan--Raspaud conjectures. The latter, a consequence of the former, states that every bridgeless cubic graph admits three perfect matchings $M_1, M_2, M_3$ such that $M_1\cap M_2\cap M_3=\emptyset$. Here we show that the Fan--Raspaud conjecture is true for bridgeless claw-free cubic graphs. This also gives further information on the interaction of perfect matchings in a class where the $5$-cycle double cover conjecture requires control of four of them.
\end{abstract}

\end{frontmatter}
\section{Introduction}
Throughout, graphs are finite and loopless, and parallel edges are allowed.
A graph is \emph{simple} if it contains no parallel edges, whilst we use the
term \emph{multigraph} when we wish to emphasise the presence of parallel edges. Let $G$ be a bridgeless cubic graph. 
For an edge set $F\subseteq E(G)$ and a vertex $v\in V(G)$, we denote the set of edges belonging to $F$ which are incident with $v$ by $\delta_F(v)$, and the number of edges belonging to $F$ which are incident with $v$ by $d_F(v)$, that is, $d_F(v)=|\delta_F(v)|$. A \emph{join} $J$ of $G$ is a subset of $E(G)$ such that $d_J(v)$ is equal to $1$ or $3$, for every $v\in V(G)$, whilst a \emph{perfect matching} $M$ of $G$ is a join for which $d_M(v)=1$, for every $v\in V(G)$.

A \emph{circuit} is a connected $2$-regular subgraph. An \emph{even subgraph} $C$ of $G$ is a subgraph such that $d_{E(C)}(v)$ is even for every $v\in V(G)$. Since $G$ is cubic, $d_{E(C)}(v)\in\{0,2\}$ for every $v\in V(G)$, and hence every non-empty even subgraph is a disjoint union of circuits, possibly of length $2$, that is, on two vertices (such a circuit is referred to as a \emph{digon}). Note that an even subgraph is not necessarily spanning.

The cycle double cover conjecture, independently proposed by several authors in various forms \cite{ItaiRodeh,SeymourCDC,SzekeresCDC}, asserts that every
bridgeless graph admits a cycle double cover, equivalently, a collection of even subgraphs such that every edge of the graph considered belongs to exactly two of them. Very recently, OpenAI \cite{OpenAICDC} announced a proof of the cycle double cover conjecture, yielding eight even subgraphs such that each edge belongs to exactly two of them; see also the expositions by Geelen \cite{GeelenCDC} and Oum \cite{OumCDC}. The corresponding assertion with five even subgraphs, referred to as the $5$-cycle double cover ($5$-CDC) conjecture \cite{Celmins5CDC, Preissmann5CDC}, remains open. The $5$-CDC conjecture is equivalent to the following statement: every bridgeless claw-free cubic graph can be covered by four perfect matchings \cite{HakMkr}.

The problem of covering the edge set of a bridgeless cubic graph with a small number of perfect matchings has received considerable attention and is closely related to several classical conjectures on cubic graphs. The Berge--Fulkerson conjecture asserts that every bridgeless cubic graph can be covered by five perfect matchings (a Berge cover). This formulation, due to Berge (unpublished), was shown by Mazzuoccolo \cite{MazzuoccoloEquivalence} to be equivalent to Fulkerson's conjecture which states that every bridgeless cubic graph admits six perfect matchings such that every edge belongs to exactly two of them (a Fulkerson cover) \cite{BergeFulkerson}.

A well-known consequence of the Berge--Fulkerson conjecture is the Fan--Raspaud conjecture \cite{FanRaspaud}, which asserts that every bridgeless cubic graph admits three perfect matchings $M_1,M_2,M_3$, referred to as an FR-triple, such that
$M_1\cap M_2\cap M_3=\emptyset$ (any three perfect matchings from a Fulkerson cover form an FR-triple). Thus, covering and intersection properties of perfect matchings are closely intertwined in the study of bridgeless cubic graphs. Further results on the Fan--Raspaud conjecture include \cite{KaiserRaspaud,MacajovaSkovieraOddness2,s4gmjp}; we also refer the reader to the theorems in \cite{quelling1,quelling2} and the conjectures in \cite{MacajovaSkovieraOddCuts,MazzuoccoloS4}, which deal with pairs of perfect matchings.

By the classical edge-colouring theorems of Vizing \cite{Vizing} and Shannon \cite{Shannon}, every cubic graph is either $3$- or $4$-edge-colourable; we shall call these Class~I and Class~II cubic graphs, respectively. If $G$ is a Class~I cubic graph, then its three colour classes are perfect matchings and taking two copies of each yields a Fulkerson cover of $G$. Thus, the Berge--Fulkerson and the Fan--Raspaud conjectures hold trivially for Class~I cubic graphs, and the problem is reduced to Class~II bridgeless cubic graphs. If Berge's conjectured bound holds, Class~II bridgeless cubic graphs can be partitioned into two classes: those that can be covered by $4$ perfect matchings and those that require $5$.

In this context, coverings by four perfect matchings occupy a natural intermediate position. Although not every Class II bridgeless cubic graph can be covered by four perfect matchings---the Petersen graph being the standard example---restricting the covering problem to claw-free cubic graphs leads to a conjecture of a very different nature. As mentioned before, the assertion that every bridgeless claw-free cubic graph can be covered by four perfect matchings is equivalent to the $5$-CDC conjecture.

Before continuing, observe also that a cover of a bridgeless cubic graph by four perfect matchings immediately yields an FR-triple, revealing a natural connection between the two notions. Motivated by this connection, we study how three perfect matchings can
interact in the claw-free setting. We prove that every bridgeless claw-free cubic graph admits an FR-triple. Besides establishing the Fan--Raspaud conjecture for this class, this may provide further insight into the interaction of four perfect matchings underlying the $5$-CDC formulation.


\section{Main result}
We first recall that, by a result of Matthews \cite{Matthews} together with Jaeger's $8$-flow theorem \cite{Jaeger1979,Jaeger1975}, the edge set of every bridgeless graph can be covered by three even subgraphs. We shall use this fact in the following lemma.

\begin{lemma}\label{lem:triangle-expansion}
Let $G$ be a bridgeless cubic graph and let $G^{\triangle}$ be obtained from $G$ by expanding each vertex of $G$ into a triangle. Then $G^{\triangle}$ admits three perfect matchings with empty intersection.
\end{lemma}

\begin{proof}
Let $\mathcal{C}_1,\mathcal{C}_2,\mathcal{C}_3$ be three even subgraphs of $G$ such that $E(G)=E(\mathcal{C}_1)\cup E(\mathcal{C}_2)\cup E(\mathcal{C}_3)$ and, for every $i\in\{1,2,3\}$, let $J_i=E(G)\setminus E(\mathcal{C}_i)$. Since $G$ is cubic and every vertex has even degree in $\mathcal{C}_i$, every vertex has degree either $1$ or $3$ in $J_i$. Thus, each $J_i$ is a join of $G$. Moreover,
\[
    J_1\cap J_2\cap J_3
    =E(G)\setminus\left(E(\mathcal{C}_1)\cup E(\mathcal{C}_2)\cup E(\mathcal{C}_3)\right)
    =\emptyset.
\]

For every $v\in V(G)$, let $T_v$ denote the triangle in $G^{\triangle}$ obtained after expanding $v$, as illustrated in Figure~\ref{fig:expansion}. Each edge in $E(G)$ naturally determines a unique edge in $E(G^\triangle)$ not contained in any triangle $T_v$; we call these corresponding edges.

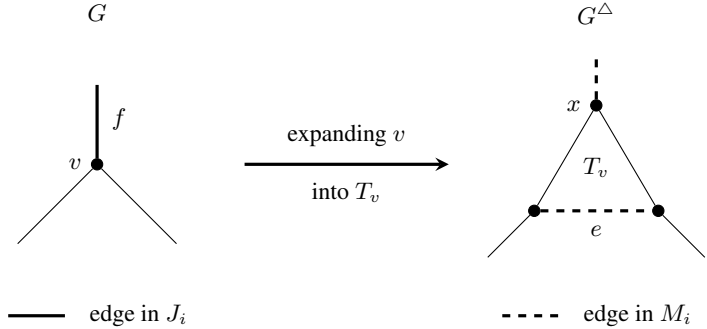
\begin{figure}[H]
\centering
\begin{tikzpicture}[
    >=Latex,
    vertex/.style={circle, draw, fill=black, inner sep=1.5pt},
    every node/.style={font=\small}
]

\begin{scope}[shift={(0,0)}]

    \path (-1.65,-1.75) rectangle (1.65,1.75);

    \node at (0,2) {$G$};

    \node[vertex, label=left:$v$] (v) at (0,0) {};

    \draw (v) -- (-1.05,-1.05);
    \draw (v) -- ( 1.05,-1.05);

    \draw[very thick] (v) -- (0,1.05);

    \node[right=2pt] at (0,0.60) {$f$};

    \node at (0,-2) {%
        \tikz[baseline=-0.5ex]
        \draw[very thick] (0,0) -- (0.75,0);
        \hspace{0.15cm} edge in $J_i$};

\end{scope}

\draw[->, >=stealth, very thick]
    (1.95,0) -- (4.65,0)
    node[midway, above=3pt] {expanding $v$}
    node[midway, below=3pt] {into $T_v$};

\begin{scope}[shift={(6.6,0)}]

    \path (-1.65,-1.75) rectangle (1.65,1.75);

    \node at (0,2) {$G^{\triangle}$};

    \begin{scope}[scale=0.82]

        \node[vertex] (a) at (-1,-0.75) {};
        \node[vertex] (b) at ( 1,-0.75) {};
        \node[vertex, label=left:$x$] (x) at (0,0.95) {};

        \draw (a) -- (x) -- (b);

        \draw (a) -- (-1.75,-1.50);
        \draw (b) -- ( 1.75,-1.50);

        \draw[very thick, dashed] (a) -- (b);
        \draw[very thick, dashed] (x) -- (0,1.75);

        \node at (0,-0.05) {$T_v$};
        \node[below=2pt] at (0,-0.75) {$e$};

    \end{scope}

\node at (0,-2) {%
        \tikz[baseline=-0.5ex]
        \draw[very thick,dashed] (0,0) -- (0.75,0);
    \hspace{0.15cm} edge in $M_i$};

\end{scope}
\end{tikzpicture}
\caption{Expanding vertex $v$ into triangle $T_v$ and extension of $J_i$ to $M_i$ when $d_{J_i}(v)=1$.}
\label{fig:expansion}
\end{figure}

We extend each $J_i$ to a perfect matching $M_i$ of $G^{\triangle}$ as follows.
If $d_{J_i}(v)=3$, all three vertices of $T_v$ are already incident with
an edge corresponding to an edge of $J_i$, and we add no edge of $T_v$
to $M_i$. If $d_{J_i}(v)=1$, exactly one vertex of $T_v$ is incident with
an edge corresponding to an edge of $J_i$, and we add to $M_i$ the edge
joining the other two vertices of $T_v$. Clearly, $M_i$ is a perfect
matching of $G^{\triangle}$.

Next, we show that $M_1\cap M_2\cap M_3=\emptyset$. Suppose otherwise, and let $e\in M_1\cap M_2\cap M_3$. If $e$ corresponds to an edge of $G$, then $e\in J_1\cap J_2\cap J_3$, a contradiction.

Consequently, $e$ belongs to some triangle $T_v$ as in Figure \ref{fig:expansion}. Let $x$ be the unique
vertex of $T_v$ not incident with $e$. Let $i\in\{1,2,3\}$. Since $M_i$ is a perfect matching of $G^{\triangle}$, $d_{M_i}(x)=1$ and $\delta_{M_i}(x)\cap E(T_v)=\emptyset$. Let $f$ be the edge of $G$ incident with $v$ corresponding to the edge in $\delta_{M_i}(x)$ of $G^{\triangle}$. The construction of $M_i$ implies that $\delta_{J_i}(v)=\{f\}$ for every $i\in\{1,2,3\}$. In particular, $f\in J_1\cap J_2\cap J_3$, again a contradiction. Hence, $M_1\cap M_2\cap M_3=\emptyset$.
\end{proof}

A \emph{diamond} is the complete graph on four vertices with an edge deleted as in Figure~\ref{fig:diamond}. The two distinguished vertices of degree two of a diamond are called the \emph{head} and \emph{tail}.

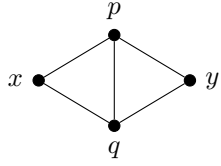
\begin{figure}[H]
    \centering
    \begin{tikzpicture}[
        vertex/.style={circle, draw, fill=black, inner sep=1.5pt}
    ]
        \node[vertex, label=left:$x$]  (x) at (0,0) {};
        \node[vertex, label=above:$p$] (p) at (1,0.6) {};
        \node[vertex, label=below:$q$] (q) at (1,-0.6) {};
        \node[vertex, label=right:$y$] (y) at (2,0) {};

        \draw
            (x) -- (p)
            (x) -- (q)
            (p) -- (q)
            (p) -- (y)
            (q) -- (y);
    \end{tikzpicture}
    \caption{A diamond with vertices $x,p,q,y$, having head $x$ and tail $y$.}
    \label{fig:diamond}
\end{figure}

Let $D_1,\ldots,D_k$ be a maximal sequence of diamonds. A \emph{string of diamonds} is obtained by joining the tail of $D_i$ to the head of $D_{i+1}$ by an edge, for every $i\in\{1,\ldots,k-1\}$. Thus, the resulting string has exactly two vertices of degree two, namely the head of $D_1$ and the tail of $D_k$, which are called the \emph{head} and \emph{tail} of the string, respectively.  A string consisting of a single diamond is simply a diamond.

Let $G$ be a bridgeless cubic graph and let $u$ and $v$ be two adjacent vertices of $G$. Replacing the edge $uv\in E(G)$ by a string of diamonds with head $x$ and tail $y$ means deleting $uv$, inserting the string of diamonds, and adding the edges $ux$ and $yv$, as illustrated in Figures~\ref{fig:string-of-diamonds} and \ref{fig:diamond-replacement}.

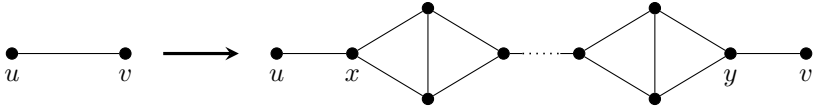
\begin{figure}[H]
    \centering
    \begin{tikzpicture}[
        vertex/.style={circle, draw, fill=black, inner sep=1.5pt}
    ]

        \node[vertex, label=below:$u$] (u1) at (0,0) {};
        \node[vertex, label=below:$v$] (v1) at (1.5,0) {};
        \draw (u1) -- (v1);

        \draw[->, >=stealth, very thick] (2,0) -- (3,0);

        \node[vertex, label=below:$u$] (u2) at (3.5,0) {};
        \node[vertex, label=below:$x$] (x)  at (4.5,0) {};

        \node[vertex] (p1) at (5.5,0.6) {};
        \node[vertex] (q1) at (5.5,-0.6) {};
        \node[vertex] (m1) at (6.5,0) {};

        \node[vertex] (m2) at (7.5,0) {};
        \node[vertex] (p2) at (8.5,0.6) {};
        \node[vertex] (q2) at (8.5,-0.6) {};
        \node[vertex, label=below:$y$] (y)  at (9.5,0) {};
        \node[vertex, label=below:$v$] (v2) at (10.5,0) {};

        \draw
            (u2) -- (x)
            (x) -- (p1)
            (x) -- (q1)
            (p1) -- (m1)
            (q1) -- (m1)
            (p1) -- (q1)
            (m2) -- (p2)
            (m2) -- (q2)
            (p2) -- (y)
            (q2) -- (y)
            (p2) -- (q2)
            (y) -- (v2);

        \draw (m1) -- (6.75,0);
        \draw[semithick, dotted] (6.75,0) -- (7.25,0);
        \draw (7.25,0) -- (m2);

    \end{tikzpicture}
    \caption{Replacing the edge $uv$ with a string of diamonds having head $x$ and tail $y$.}
    \label{fig:string-of-diamonds}
\end{figure}

Let $\mathcal{E}\subseteq E(G)$. The graph obtained from $G$ by replacing each edge in $\mathcal{E}$ by a string of diamonds is denoted by $G^{\diamond(\mathcal{E})}$. If more than one string of diamonds is used in the construction of $G^{\diamond(\mathcal{E})}$, that is, if $|\mathcal{E}|>1$, the strings need not have the same length.

We next show that replacing an edge of a bridgeless cubic graph admitting an FR-triple by a diamond preserves the existence of an FR-triple.

\begin{lemma}\label{lem:diamond}
Let $G$ be a bridgeless cubic graph admitting three perfect matchings $M_1,M_2, \linebreak M_3$ such that $M_1\cap M_2\cap M_3=\emptyset$. Let $uv\in E(G)$. The graph $G^{\diamond(\{uv\})}$ obtained from $G$ by replacing the edge $uv$ with a diamond also admits three perfect matchings with empty intersection.
\end{lemma}

\begin{proof}
Let $uv\in E(G)$ be the edge to be replaced by a diamond. Denote the two vertices of degree $2$ in the diamond by $x$ and $y$, and the other two vertices by $p$ and $q$, where $u$ is adjacent to $x$ and $v$ is adjacent to $y$. Thus, the edges of $G^{\diamond(\{uv\})}$ replacing $uv$ are $ux, xp, xq, pq, yp, yq, yv$ (see Figure \ref{fig:diamond-replacement}).

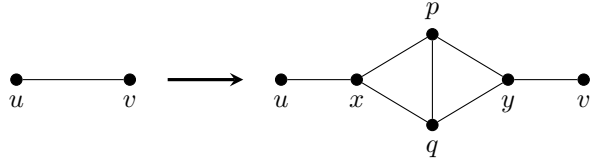
\begin{figure}[H]
    \centering
    \begin{tikzpicture}[
        vertex/.style={circle, draw, fill=black, inner sep=1.5pt}
    ]

        \node[vertex, label=below:$u$] (u1) at (0.5,0) {};
        \node[vertex, label=below:$v$] (v1) at (2,0) {};
        \draw (u1) -- (v1);

        \draw[->, >=stealth, very thick] (2.5,0) -- (3.5,0);

        \node[vertex, label=below:$u$] (u2) at (4,0) {};
        \node[vertex, label=below:$x$] (x)  at (5,0) {};
        \node[vertex, label=above:$p$] (p)  at (6,0.6) {};
        \node[vertex, label=below:$q$] (q)  at (6,-0.6) {};
        \node[vertex, label=below:$y$] (y)  at (7,0) {};
        \node[vertex, label=below:$v$] (v2) at (8,0) {};

        \draw
            (u2) -- (x)
            (x) -- (p)
            (x) -- (q)
            (p) -- (q)
            (p) -- (y)
            (q) -- (y)
            (y) -- (v2);

    \end{tikzpicture}
    \caption{Replacing the edge $uv$ with a diamond having head $x$ and tail $y$.}
    \label{fig:diamond-replacement}
\end{figure}

Let $P=\{ux,pq,yv\}$, $A=\{xp,yq\}$, and $B=\{xq,yp\}$. For each $i\in\{1,2,3\}$ such that $uv\in M_i$, replace $uv$ in $M_i$ with
the three edges in $P$. If $uv\notin M_i$, extend $M_i$ inside the
diamond using either $A$ or $B$. In either case a perfect matching of
$G^{\diamond(\{uv\})}$ is obtained. Let $N_i$ be the resulting perfect matching of $G^{\diamond(\{uv\})}$.

Let $\mathcal{I}=\{i\in\{1,2,3\}:uv\in M_i\}$. Since $M_1\cap M_2\cap M_3=\emptyset$, we have $|\mathcal{I}|\leq 2$. We consider two cases.  

\textbf{Case 1.} If $\mathcal{I}\neq\emptyset$, for each $i\notin \mathcal{I}$, we extend $M_i$ to $N_i$ by using the edges from $A$. Every edge of $P$ then belongs to precisely $|\mathcal{I}|\leq2$ of the resulting perfect matchings $N_1, N_2, N_3$, whereas every edge of $A$ belongs to at most $3-|\mathcal{I}|\leq2$ of them. No edge of $B$ is used when extending $M_1, M_2, M_3$ to $N_1, N_2, N_3$. Consequently, no newly introduced edge belongs to all three of $N_1, N_2, N_3$, and so $N_1\cap N_2\cap N_3=\emptyset$.

\textbf{Case 2.} If $\mathcal{I}=\emptyset$, we respectively extend $M_1$ and $M_2$ to $N_1$ and $N_2$ by using the edges from $A$, and extend $M_3$ to $N_3$ by using the edges from $B$. Since $A\cap B=\emptyset$, once again no edge of the
diamond belongs to $N_1\cap N_2 \cap N_3$. Moreover, outside the diamond the three perfect matchings coincide with $M_1,M_2,M_3$, and so their common intersection is empty. Again, $N_1\cap N_2\cap N_3=\emptyset$.

Cases 1 and 2 imply that $G^{\diamond(\{uv\})}$ admits an FR-triple.
\end{proof}

A \emph{claw} is the complete bipartite graph $K_{1,3}$ with partite sets of cardinality $1$ and $3$, and a graph is \emph{claw-free} if it contains no induced copy of $K_{1,3}$. The class of simple bridgeless claw-free cubic graphs admits a well-known structural description due to Oum \cite{Oum2011}. Namely, every connected simple bridgeless claw-free cubic graph is of one of the
following three types.
\begin{enumerate}[label=(\roman*)]
    \item The complete graph on four vertices $K_4$.
    \item A ring of diamonds (see Figure~\ref{fig:ring-of-diamonds}) containing at least two diamonds. Note that a connected claw-free cubic graph in which every vertex belongs to an induced diamond is a ring of diamonds.
    \item A graph obtained from a bridgeless cubic graph $G$ by:
    \begin{itemize}
        \item first expanding each vertex $v\in V(G)$ into a triangle $T_v$, thus obtaining \(G^{\triangle}\); and 
        \item then replacing each edge in a set $\mathcal{E}\subseteq E(G^{\triangle})\setminus\bigcup_{v\in V(G)}E(T_v)$ by a string of diamonds.
    \end{itemize}    
    Thus, $\mathcal{E}$ consists only of edges of $G^{\triangle}$ which correspond to edges of $G$, and contains no edge belonging to one of the newly introduced triangles. In this sense, we denote the resulting simple bridgeless claw-free cubic graph by $G^{\triangle,\diamond(\mathcal{E})}$.
\end{enumerate}

\begin{figure}[H]
    \centering
    \begin{tikzpicture}[
        vertex/.style={circle, draw, fill=black, inner sep=1.5pt}
    ]

        \node[vertex] (t1) at (-3,0) {};
        \node[vertex] (l1) at (-2,-0.6) {};
        \node[vertex] (r1) at (-2, 0.6) {};
        \node[vertex] (b1) at (-1,0) {};

        \draw
            (t1) -- (l1)
            (t1) -- (r1)
            (l1) -- (r1)
            (l1) -- (b1)
            (r1) -- (b1);

        \node[vertex] (t2) at (0,0) {};
        \node[vertex] (l2) at (1,-0.6) {};
        \node[vertex] (r2) at (1, 0.6) {};
        \node[vertex] (b2) at (2,0) {};

        \draw
            (t2) -- (l2)
            (t2) -- (r2)
            (l2) -- (r2)
            (l2) -- (b2)
            (r2) -- (b2);

        \draw (b1) -- (t2);

        \draw
            (t1).. controls (-3,1.2) and (-2.4,2.2) .. (-1.6,2.2);

        \draw[dashed]
            (-1.6,2.2).. controls (-0.9,2.2) and (-0.1,2.2) ..(0.6,2.2);

        \draw(0.6,2.2).. controls (1.4,2.2) and (2,1.2) ..(b2);

    \end{tikzpicture}
    \caption{A ring of diamonds.}
    \label{fig:ring-of-diamonds}
\end{figure}
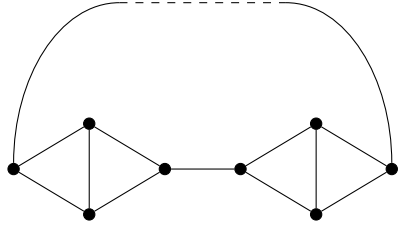

Oum's characterisation above is stated for simple cubic graphs. Hakobyan and Mkrtchyan \cite{HakobyanMkrtchyanSylvester} extended this characterisation to cubic multigraphs. In this setting, $K_2^3$, consisting of two vertices joined by three parallel edges, appears alongside $K_4$ as an additional case, whilst digons (circuits of length two) may occur alongside diamonds in the corresponding strings and rings. We shall return to this extension after proving the main result for simple graphs.

\begin{theorem}\label{thm:claw-free-FR}
Every simple bridgeless claw-free cubic graph admits an FR-triple.
\end{theorem}

\begin{proof}
Let $G$ be a simple bridgeless claw-free cubic graph. We may assume that $G$ is connected. If $G$ is the complete graph on four vertices or a ring of diamonds, then $G$ is $3$-edge-colourable, and so the three colour classes form three pairwise disjoint perfect matchings. Hence, we may assume that $G$ is of Type (iii) according to Oum's classification given above. This means that $G$ is some $G_0^{\triangle, \diamond(\mathcal{E})}$ for some bridgeless cubic graph $G_0$ and for some $\mathcal{E}\subseteq E(G_0^{\triangle})\setminus\bigcup_{v\in V(G_0)}E(T_v)$. 


By Lemma~\ref{lem:triangle-expansion}, $G_0^{\triangle}$ admits three perfect matchings with empty intersection. Moreover, $G$ can be obtained by repeatedly replacing an edge with a diamond. By Lemma~\ref{lem:diamond}, after each such replacement the resulting graph still admits three perfect matchings with empty intersection.

Therefore, $G$ admits three perfect matchings $M_1,M_2,M_3$ such that $M_1\cap M_2\cap M_3=\emptyset$, and the result follows.
\end{proof}

\subsection{Extension to multigraphs}

We conclude by showing that the result extends to cubic multigraphs. Recall that a \emph{digon} (or $2$-circuit) consists of two vertices joined by two parallel edges. Let $G$ be a cubic graph and let $uv\in E(G)$. By replacing $uv$ by a digon we mean deleting $uv$, introducing two new vertices $x$ and $y$ joined by two parallel edges $e_1$ and $e_2$, and adding the edges $ux$ and $yv$, as illustrated in Figure \ref{fig:digon-replacement}. The notions of head and tail extend naturally to digons.

\begin{figure}[H]
    \centering
    \begin{tikzpicture}[
        vertex/.style={circle, draw, fill=black, inner sep=1.5pt}
    ]

        \node[vertex, label=below:$u$] (u1) at (0.5,0) {};
        \node[vertex, label=below:$v$] (v1) at (2,0) {};
        \draw (u1) -- (v1);

        \draw[->, >=stealth, very thick] (2.5,0) -- (3.5,0);

        \node[vertex, label=below:$u$] (u2) at (4,0) {};
        \node[vertex, label=below:$x$] (x)  at (5,0) {};
        \node[vertex, label=below:$y$] (y)  at (7,0) {};
        \node[vertex, label=below:$v$] (v2) at (8,0) {};

        \draw (u2) -- (x);
        \draw (y) -- (v2);

        \draw[bend left=30] (x) to node[midway, above] {$e_1$} (y);
        \draw[bend right=30] (x) to node[midway, below] {$e_2$} (y);

    \end{tikzpicture}
    \caption{Replacing the edge $uv$ with a digon having vertices $x$ and $y$.}
    \label{fig:digon-replacement}
\end{figure}
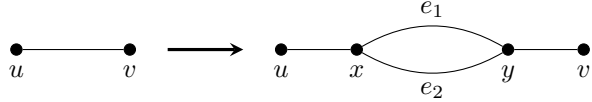

\begin{lemma}\label{lem:digon-replacement}
Let $G$ be a cubic graph admitting an FR-triple, and let $G'$ be obtained from $G$ by replacing an edge by a digon. Then $G'$ also admits an FR-triple.
\end{lemma}

The proof of Lemma \ref{lem:digon-replacement} follows by a simple argument, which we include for completeness. Let $M_1,M_2,M_3$ be an FR-triple of $G$, and suppose that $uv\in E(G)$ is replaced by a digon as in Figure \ref{fig:digon-replacement}. We extend the perfect matchings $M_i$ to perfect matchings of $G'$ as follows. For each $i\in\{1,2,3\}$, if $uv\in M_i$, replace $uv$ by $ux$ and $yv$; otherwise, extend $M_i$ by one of $e_1,e_2$. Since $uv\notin M_1\cap M_2\cap M_3$, at most two of the resulting perfect
matchings contain $ux$ and $yv$. If $uv$ belongs to none of the $M_i$, choose the extensions so that the same edge amongst $e_1$ and $e_2$ is not used by all three perfect matchings. The resulting three perfect matchings form an FR-triple of $G'$, proving Lemma \ref{lem:digon-replacement}.

A \emph{string of diamonds and digons} is a sequence $D_1,\ldots,D_k$, where each $D_i$ is either a diamond or a digon, such that consecutive members are joined in the usual way. A \emph{ring of diamonds and digons} is defined analogously.

Hakobyan and Mkrtchyan \cite{HakobyanMkrtchyanSylvester} extended Oum's characterisation to cubic graphs which may contain parallel
edges. In the bridgeless case \cite{Mkrtchyan2025}, every connected claw-free cubic graph belongs to one of the following types.
\begin{enumerate}[label=(\roman*)]
    \item The graph $K_4$ or the graph $K_2^3$, where the latter consists of two vertices joined by three parallel edges.

    \item A ring of diamonds and digons containing at least two members.

    \item A graph obtained from a connected bridgeless cubic graph $G$ by:
    \begin{itemize}
        \item first expanding each vertex $v\in V(G)$ into a triangle $T_v$, thus obtaining $G^{\triangle}$; and
        \item then replacing each edge in a set $\mathcal{E}\subseteq E(G^{\triangle})\setminus\bigcup_{v\in V(G)}E(T_v)$ by a string of diamonds and digons.
    \end{itemize}
\end{enumerate}

We now turn to the proof of our final result. The reverse operation of that shown in Figure \ref{fig:digon-replacement} is called \emph{smoothing} of a digon---it consists of deleting the two vertices of the digon and restoring the edge $uv$. This operation shall be used in what follows.

\begin{corollary}
Every bridgeless claw-free cubic graph admits an FR-triple.
\end{corollary}

\begin{proof}
By Theorem~\ref{thm:claw-free-FR}, the result holds in the simple case, so it suffices to consider a bridgeless claw-free cubic multigraph $G$. We may assume that $G$ is connected. By the characterisation due to Hakobyan and Mkrtchyan, repeatedly smoothing digons reduces $G$ either to a simple bridgeless claw-free cubic graph or to $K_2^3$. In the former case, the resulting graph admits an FR-triple by Theorem~\ref{thm:claw-free-FR}. In the latter case, the result is immediate, since the three edges of $K_2^3$ are three pairwise disjoint perfect matchings. The original graph $G$ can then be recovered by successively replacing edges by digons. By Lemma~\ref{lem:digon-replacement}, each such replacement preserves the existence of an FR-triple. Hence $G$ admits an FR-triple.
\end{proof}

\end{document}